\documentclass[12pt]{article}
\usepackage{amsmath, amsfonts, amssymb, amsthm}
\usepackage[colorlinks = true, linkcolor = blue]{hyperref}
\newtheorem{theorem}{Theorem}[section]
\newtheorem{proposition}[theorem]{Proposition}

\newtheorem{remark}[theorem]{Remark}

\newtheorem{corollary}[theorem]{Corollary}
\newtheorem{lemma}[theorem]{Lemma}

\def\R{\mathbb R}

\def\R{\mathbb{R}}

\def \exp {{\rm exp}}

\def \R {\mathbb{R}}

\begin{document}
	
	\title{Dense images of the power maps for a disconnected real algebraic group}
	
	\author{Arunava Mandal}

	\maketitle
	\begin{abstract}  
		Let $G$ be a complex algebraic group defined over $\mathbb R$, which is not necessarily Zariski connected.
		In this article, we study the density of the images of the power maps $g\to g^k$, $k\in\mathbb N$, on real points 
		of $G$, i.e., $G(\mathbb R)$ equipped with the real topology. 
		As a result, we extend a theorem of P. Chatterjee 
		on surjectivity of the power map for the set of semisimple elements of $G(\mathbb R)$. 
		We also characterize surjectivity of the power map for a disconnected group $G(\mathbb R)$.
		The results are applied in particular to describe the image of the exponential map of $G(\mathbb R).$  
		
	\end{abstract} 
	\noindent {\it Keywords}: {Power maps, real algebraic groups, weak exponentiality.}

\section{Introduction}

Let $G$ be a topological group. For $k\in\mathbb N$, let $P_k:G\to G$ be the $k$-th power map of $G$ defined by
$P_k(g)=g^k$ for all $g\in G$. 
This article is mainly concerned with the question as to when such a map has a {\it dense} image or a {\it surjective} image
for real points of an algebraic group.

There is a vast amount of literature with regard to analogous questions in
the case of exponential maps of (connected) Lie groups (see \cite{DH}, \cite{HM}, \cite{HL} and \cite{DT} for example).
Moreover, it is known,
proved independently by K. Hofmann, J. Lawson  \cite{HL} and  M. McCrudden \cite{Mc} that 
the exponential map of a {\it connected} Lie group $G$ is surjective
if and only if $P_k(G)=G$ for all $k\in\mathbb N$.
Recently, it was proved
that the exponential map has a dense image in $G$ if and only if the image of the $k$-th power map is dense
for all $k$ (see \cite{BM}). This gives a motivation to study the images of the power maps.

A well-known result of A. Borel states that the image of a
word map (and hence in particular
a power map) on a semisimple {\it algebraic group} $G$ is Zariski-dense (see \cite{B1}).
Note that the image of the power map is not always dense
for real points of an algebraic group in {\it real topology}. For example, in case of ${\rm SL}(2,\mathbb R)$, 
the image of $P_2$
is not dense in real topology but its image is Zariski-dense in ${\rm SL}(2,\mathbb R)$.
Therefore our aim is to study the density property of $P_k$ for real points of an algebraic group in {\it real topology}.
The density of $P_k$
is well understood for {\it connected} Lie groups (see \cite{BM}, \cite{M}). 
There has also been a considerable amount of work 
with regard to the surjectivity
of the map for both connected and {\it disconnected} groups 
(cf. e.g. \cite{Ch}, \cite{Ch1}, \cite{Ch2}, \cite{St}, \cite{DM} and also see references therein).
However, the density of the power map 
is not known for a disconnected (real) algebraic group.

In this context, we obtain results for density of the image of $P_k$ on  
real points $G(\mathbb R)$ of a complex algebraic group $G$, which is defined over $\mathbb R$
and is {\it not necessarily Zariski connected} 
(see Theorem \ref{T1}). Indeed, Theorem \ref{T1} provides an extension of \cite[Theorem 5.5]{Ch1}, 
which deals with the surjectivity of the power map on the set of semisimple elements of $G(\mathbb R)$.
As a consequence of Theorem \ref{T1}, we get a characterization of the {\it surjectivity} of the power map for
(possibly disconnected) group
$G(\mathbb R)$ (see Corollary \ref{C3}).

Let $G$ be a complex algebraic group defined over $\mathbb R$. Let 
$G(\mathbb R)$ be the set of $\mathbb R$-points of $G$. Note that $G(\mathbb R)$ is equipped with the
real topology and has a real manifold structure. We denote the Zariski connected component of identity of $G$ by $G^0$. Let
$G(\mathbb R)^*$ be the connected component of identity of $G(\mathbb R)$ in Hausdorff or real topology.
We denote $(a,b)=1$ if two integers $a$ and $b$ are co-prime. Also, the order of a finite group $F$ is denoted by $\circ(F)$.
For a subset $A$ of an algebraic group $G$, we denote the set of semisimple elements of $A$ by $S(A)$.

The following is the main result of this article. It characterizes the density of the images of the power maps for a disconnected algebraic group which is not necessarily reductive, and provides an extension of \cite[Theorem 5.5]{Ch1}.
In the context of surjectivity of the power map, an analogous result 
(corresponding to $(1)\Leftrightarrow (2)$ in Theorem \ref{T1}) was proved by P. Chatterjee 
(see \cite[Theorem 1.8]{Ch1}). Also, he showed $(3)\Leftrightarrow (4)$ in Theorem \ref{T1}
(see \cite[Theorem 5.5]{Ch1}).

\begin{theorem}\label{T1}
	Let $G$ be a complex algebraic group, not necessarily Zariski-connected, defined over 
	$\mathbb R$.
	Let $A$ be a subgroup of $G$ with  $G(\mathbb R)^*\subset A\subset G(\mathbb R)$ and
	$k\in\mathbb N$. Then the following are equivalent.
	\begin{enumerate}
	\item $P_k(A)$ is dense in $A$.
	
	\item $(k,\circ(A/G(\mathbb R)^*))=1$
	and $P_k(G(\mathbb R)^*)$ is dense.
	
	\item $(k,\circ(A/G(\mathbb R)^*))=1$
	and $P_k:S(G(\mathbb R)^*)\to S(G(\mathbb R)^*)$ is surjective.
	
	\item $P_k:S(A)\to S(A) $ is surjective.
	\end{enumerate}
	In particular, if $G$ is a Zariski connected algebraic group defined over $\mathbb R$ 
	then for an odd $k\in\mathbb N$, $P_k(G(\mathbb R))$ is dense in $G(\mathbb R)$ 
	if and only if $P_k(G(\mathbb R)^*)$ is dense in $G(\mathbb R)^*$.
	\end{theorem}

The connection of the power maps with weak exponentiality is discussed in \S 4 and \S 5. The next result (see \S 5) characterizes the surjectivity of the power map for a disconnected group $G(\mathbb R)$. 

\begin{corollary}\label{C3}
	Let $G$ be an algebraic group defined over $\mathbb R$, which is not necessarily Zariski connected.
	Let $k\in\mathbb N$. 
	Then $P_k:G(\mathbb R)\to\ G(\mathbb R)$ is surjective if and only if
	$P_k(Z_{G(\mathbb R)}(u))$ is dense in $Z_{G(\mathbb R)}(u)$ for any unipotent element $u\in G(\mathbb R)$.
\end{corollary}

An application of Corollary \ref{C3} is given in Corollary \ref{C14}, which
characterizes the exponentiality of $G(\mathbb R)$. Corollary \ref{C3} can also be thought as an extension of \cite[Theorem 1.7]{Ch1}.

\section{Preliminaries}

Let $G$ be a complex algebraic group.
Let $G^0$ denote
the Zariski-connected component of the identity of $G$.
The maximal, Zariski connected, Zariski closed, normal unipotent subgroup of
$G$ is called the unipotent radical of $G$ and is denoted by $R_u(G)$.
Note that $R_u(G)=R_u(G^0).$

We now include the following results which will be used subsequently:

\begin{theorem}{\rm (\cite[Theorem 8.9(c)]{HP})}\label{A2}
	Let $G$ be a complex algebraic group, which is not necessarily connected.
	Suppose the connected component $G^\circ$ is reductive. Then for each 
	$a \in G$ the set of semisimple elements  $S(G) \cap G^\circ a$ 
	is a non-empty Zariski open subset of the coset $G^\circ a$. In particular 
	$S(G)$ is Zariski dense in $G$.
\end{theorem}
\begin{proposition}\label{B}
	Let $G$ be a complex algebraic group, not necessarily Zariski connected, defined
	over $\mathbb R$ such that $G^0$ is reductive. Then $S(G(\mathbb R))$ contains an open dense subset of $G(\mathbb R)$ in the
	Hausdorff topology of $G (\mathbb R)$.
\end{proposition}
\begin{proof}
	Since $G (\mathbb R)/ G^0(\mathbb R)$ embeds in $G/G^0$, it follows that
	$G (\mathbb R) / G^0(\mathbb R)$ is a finite group. Let $g_1, \ldots, g_p \in G(\mathbb R)$ such that
	$$G (\mathbb R) = g_1 G^0(\mathbb R) \cup \cdots \cup g_p G^0(\mathbb R).$$
	It is enough to prove that for all $i = 1,\ldots,p$, the set $S(g_i G^0(\mathbb R))$ contains
	an open dense subset of $g_i G^0 (\mathbb R)$ in the Hausdorff topology of $G (\mathbb R)$.
	By Theorem \ref{A2}, there is a Zariski open set $W$ of $g_i G^0$  consisting of semisimple
	elements of $G$. 
	
	Let $X_{\rm sm}$ be the set of smooth points of an
	irreducible affine variety $X$ defined over $\mathbb R$,
	and $M:= X_{\rm sm} (\mathbb R) \neq \emptyset$.
	Let $W \neq \emptyset$ be a Zariski open set of $X$. Then $W \cap M$ is an
	open dense subset of $M$ in the Hausdorff topology of $M$.

	Now observe that  $ {g_i G^0}_{\rm sm} (\mathbb R) = g_i G^0 (\mathbb R)$.
	Hence by the above fact, it follows that $W \cap g_i G^0 (\mathbb R)$ is open dense in $g_i G^0(\mathbb R)$. 
	\end{proof}

We recall that for a connected Lie group $G$, an element $g\in G$ is said to be {\it regular} if the nilspace $N({\rm Ad}_g-I)$
is of minimal dimension (see \cite{Bou}). An element $g$ is $P_k$-{\it regular} if $(dP_k)_g$ is non-singular.

\begin{remark}\label{Regular}
	Let $G$ be a complex algebraic group defined over $\mathbb R$. Then any element $g\in G(\mathbb R)$ can be written as $g=g_s g_u$ where $g_s, g_u\in G(\mathbb R)$ and $g_s$ is semisimple, $g_u$ is unipotent, and $g,g_s,g_u$ commute with each other. Then we have $Ad_g=Ad_{g_s}Ad_{g_u}$. Since all the eigenvalues of $Ad_{g_u}$ are $1$, the generalized eigenspace for eigenvalue $1$ of $Ad_g$ is the same as that of $Ad_{g_s}$. Note that $N(Ad_{g_s}-I)=\ker(Ad_{g_s}-1)$ as $Ad_{g_s}$ is semisimple. The dimension of $\ker(Ad_{g_s}-1)$ is equal to the dimension of the centralizer $Z_{G(\mathbb R)}(g_s)$. So 
	$g$ is regular in $G(\mathbb R)$ if and only if 
	$Z_{G(\mathbb R)}(g_s)$ has minimal dimension. In particular, $g$ is regular in $G(\mathbb R)$ if and only if $g_s$ is regular in $G(\mathbb R)$. 
\end{remark}

 \begin{remark}
 	The definition of a regular element, mentioned above, 
 	coincide with the definition of a regular element in \S 12.2 of \cite{B} for the algebraic case. 
 	In literature, there is another notion of regular element of an algebraic group (see \cite{S}). 
 	Let $G$ be an algebraic group
 	over an algebraically closed field. An element $g\in G$ is called regular, if ${\rm dim}Z_G(g)\leq{\rm dim}Z_G(x)$ for all
 	$x\in G.$ Note that according to this definition, a unipotent element may be a regular element; 
 	in fact, if $G$ is connected and semisimple, then it admits a regular unipotent element (see \cite{S}).
 	However, our notion of regular element is different from that. For a semisimple group $G$ as above, it is easy to see that 
 	unipotent elements are not regular in our sense. Indeed, let $u$ be a unipotent element of a semisimple algebraic group $G$, then
 	${\rm Ad}_u$ is unipotent, and hence the nilspace $N({\rm Ad}_u-I)$ is the whole space, which is not of minimal dimension.
 	
 \end{remark}

To prove Theorem \ref{T1}, we need the following lemma.

\begin{lemma}\label{L1}
	Let $G$ be a complex algebraic group, not necessarily Zariski-connected, defined over $\mathbb R$. 
	Let $k\in\mathbb N$. Then the following are equivalent.
	\begin{enumerate}
		\item $P_k(S({\rm Reg}(G(\mathbb R)^*)))\supset S({\rm Reg}(G(\mathbb R)^*))$.
		
		\item $P_k:S(G(\mathbb R)^*)\to S(G(\mathbb R)^*)$ is surjective.
		
		\item The image of the map $P_k:G(\mathbb R)^*\to G(\mathbb R)^*$ is dense.
	\end{enumerate}
\end{lemma}
\begin{proof}
	First, we will show that each of $(1)$ and $(2)$ implies $(3)$.
	In view of \cite[Theorem 1.1]{BM}, it is enough to show that ${\rm Reg}(G(\mathbb R)^*)\subset P_k(G(\mathbb R)^*)$.
	Let $g\in {\rm Reg}(G(\mathbb R)^*)$. Let $g=g_sg_u$ where $g_s$ and $g_u$ are respectively 
	the semisimple part and unipotent part
	of the Jordan decomposition of $g$. By Remark \ref{Regular}, $g$ is regular if and only if $g_s$ is regular.
	Hence, either $(1)$ or $(2)$ implies that there exists $h_s\in S({\rm Reg}(G(\mathbb R)^*))$ ($h_s\in S(G(\mathbb R)^*)$) such that $h_s^k=g_s$.
	Since $g_u$ is unipotent, there exists a unique unipotent element $h_u$ such that $h_u^k=g_u$.
	As the Zariski closure of the
	cyclic subgroups generated by $g_u$ and $h_u$ are the same, $g_s$ commutes with $h_u$. 
	
	Since $h_u$ is unipotent, there exists a unique nilpotent element $X\in\mathfrak {\rm Lie}(G(\mathbb R)^*)$ such
	that $h_u=\exp X$. As $g_s$ commutes with $h_u$, we have
	$g_s\exp X {g_s}^{-1}=\exp X$, which gives $$\exp(Ad_{g_s}(X))=\exp X.$$ Now by the uniqueness of the element
	$X\in\mathfrak {\rm Lie}(G(\mathbb R)^*)$, we get $Ad_{g_s}(X)= X$.
	Therefore, $Ad_{h_s^k}(X)= X$. By \cite[Lemma 2.1]{BM},
	$h_s$ is regular and $P_k$-regular. This implies $Ad_{h_s}(X)= X$, which in turn shows that
	$h_s$ commutes with $h_u$. Hence, we have $$g=g_sg_u=h_s^kh_u^k=(h_sh_u)^k.$$
	
	Now suppose that $(3)$ holds. Let $g\in S(G(\mathbb R)^*)$ or $S({\rm Reg}(G(\mathbb R)^*))$. Then there exists a Cartan subgroup $C\subset G(\mathbb R)^*$ such that $g\in C$. By \cite[Theorem 1.1]{BM}, we have $P_k(C)=C$. Thus there exists $h\in C$ such that $h^k=g$.
	Since $g$ is semisimple (regular), $h$ is also semisimple (regular), which proves the lemma.
\end{proof}

\begin{proposition}\label{disconnected reductive}
	Let $G$ be a complex reductive algebraic group, not necessarily Zariski-connected, defined over $\mathbb R$. 
	Let $k\in\mathbb N$. Then the following are equivalent.
	\begin{enumerate}
		\item The image of $P_k:G(\mathbb R)\to G(\mathbb R)$ is dense. 
		
		\item $P_k: G(\mathbb R)/G(\mathbb R)^*\to G(\mathbb R)/G(\mathbb R)^*$ is surjective and the image of $P_k:G(\mathbb R)^*\to G(\mathbb R)^*$ is dense.
		
		\item $P_k: G(\mathbb R)/G^0(\mathbb R)\to G(\mathbb R)/G^0(\mathbb R)$ is surjective and the image of $P_k:G^0(\mathbb R)\to G^0(\mathbb R)$ is dense.
	\end{enumerate}
\end{proposition}
\begin{proof} This follows by applying  \cite[Theorem 5.5]{Ch1} and Lemma \ref{L1} so we omit the details.
\end{proof}

\section{Proof of Theorem \ref{T1}}

We use the following additional lemma to prove Theorem \ref{T1}.

\begin{lemma}\label{surjectivity of abelian subgroup}
	Let $G$ be a complex reductive algebraic group, not necessarily Zariski-connected, defined over $\mathbb R$. Let $k\in\mathbb N$. Let $G(\R)^*\subset A\subset G(\R)$.
	Suppose that $P_k(A)$ is dense in $A$. Then
	for any $s\in S(A)$ there exists a closed abelian subgroup $B_s$ (containing $s$) of $A$ with finitely many connected components such that $P_k:B_s\to B_s$ is surjective.
\end{lemma}
\begin{proof}
	In view of Propositions \ref{B} and \ref{disconnected reductive}, we note that $P_k(A)$ is dense in $A$ if and only if $P_k:S(A)\to S(A)$ is surjective. By \cite[Theorem 5.5]{Ch1}, $P_k:S(A)\to S(A)$ is surjective if and only if $k$ is co-prime to the order of $A/G(\mathbb R)^*$ and $P_k:S(G(\mathbb R)^*)\to S(G(\mathbb R)^*)$ is surjective. We conclude our assertion by following the arguments of the proof of \cite[Theorem 5.5(1)]{Ch1}. 
	
	Let $s\in S(A)$. If the order of $A/G(\mathbb R)^*$ is $m$ then $s^m\in G(\mathbb R)^*$. On page-230 of \cite[Theorem 5.5]{Ch1}, it is shown that there exists a maximal $\mathbb R$-torus $T$ of $G'=Z_G(s^m)$ such that $sTs^{-1}=T$ and $P_k:Z_{T(\mathbb R)^*}(s)\to
	Z_{T(\mathbb R)^*}(s)$ is surjective. Moreover, it is shown in Case 1 and Case 2 of \cite[Theorem 5.5]{Ch1} that $s=(t^cs^d)^k$ for some $c,d\in\mathbb Z$, where $t\in Z_{T(\mathbb R)^*}(s)$. Also, $s^n\in Z_{T(\R)^*}(s)$ for some positive integer $n$.
	
	Now we consider the subgroup $H$ of $A$ generated by $Z_{T(\mathbb R)^*}(s)$ and the element $s$. Then $H$ is a closed abelian subgroup of $A$ and $H/H^*$ is finite. Further, $P_k:H\to H$ is surjective as $P_k:Z_{T(\mathbb R)^*}(s)\to
	Z_{T(\mathbb R)^*}(s)$ is surjective and $s=(t^cs^d)^k\in H$. Now set $B_s:=H$.
\end{proof}

Next we state a result from \cite{DM} (see also \cite{Da} and \cite{Da1}) which we use in the proof of Theorem \ref{T1}.

Let $G$ be a Lie group and $N$ be a simply connected nilpotent normal subgroup of $G$. Let $A=G/N$ be abelian (not necessarily connected). 
Let $N_j$, $j=0,1,\ldots,r$, be closed connected normal subgroups of $G$ contained in $N$ such that
$$N=N_0\supset N_1\supset\cdots\supset N_r=\{e\}$$ 
with $[N,N_j]\subset N_{j+1}$. For $j\in\{0,1,\ldots,r-1\}$, let $V_j=N_j/N_{j+1}$. Then $V_j$  is a finite dimensional vector space over $\mathbb R$ for each $j$ and the $G$-action on $N$ by conjugation induces an action on $V_j$ for all $j$. Moreover, the restriction of the action to $N$ is trivial, and hence it induces an action of $G/N$ on each $V_j$, $j\in\{0,1,\ldots,r-1\}$.
The existence of such a series can be ensured by taking the central series of $N$.
We refine the sequence
$N_j$ by inserting more terms in between if necessary, and assume that the $G$-action (and hence the $G/N$-action) on $N_j/N_{j+1}$ is irreducible for all $j$.

	\begin{theorem}\rm{(\cite{DM})}\label{DaM}
		Let $G$, $N$, $A$ be as above. Let $N=N_0\supset N_1\supset\cdots\supset N_r=\{e\}$ be a series of closed connected normal subgroups $N_j$ of $N$ as above such that the action of $A=G/N$ on the vector space $V_j=N_j/N_{j+1}$ (induced by conjugation) is irreducible for all $j\in\{0,1,\ldots,r-1\}$.
	 Let $x\in G$ and $a=xN\in A$.
	 Suppose there exists $b\in A$ with $b^k=a$ such that for any $V_j$, if the action of $a$ on $V_j$ is trivial, then the action of $b$ on $V_j$ is trivial. Then for $n\in N$ there exists $y\in G$ such that $yN=b$ and $y^k=xn$.
\end{theorem}
The proof is immediate from Corollary 5.2 and Theorem 1.1(i) of \cite{DM}.\\

We now give a proof of Theorem \ref{T1}.\\

\noindent{\it Proof of Theorem~\ref{T1}}:
$(1)\Rightarrow(2):$ This is obvious.

$(2)\Rightarrow(1):$ We first give a brief outline of the proof and later explain the steps.
In step 1, we assume $G(\mathbb R)=L(\mathbb R)\ltimes R_u(G)(\mathbb R)$ for some Levi subgroup $L$ (of $G$) defined over $\mathbb R$ and consider a dense set $D=S(A\cap L(\mathbb R))R_u(G)(\mathbb R)$ of $A$. By using the hypothesis, we see that $P_k:S(A\cap L(\mathbb R))\to S(A\cap L(\mathbb R))$ is surjective, i.e., $P_k(A\cap L(\mathbb R))$ is dense. In step 2, to each element $x_0\in S(A\cap L(\mathbb R))$ we associate a subgroup $B_{x_0}$ of $A$ such that $B_{x_0}$ is abelian with finitely many components and $P_k:B_{x_0}\to B_{x_0}$ is surjective. Then we construct a solvable subgroup $G_{x_0}$ of $A$ given by $G_{x_0}=B_{x_0}\ltimes R_u(G)(\mathbb R)$.  We show there exists a dense open set $W_{x_0}$ of $G_{x_0}$ such that each element of $W_{x_0}$ has $k$-th root in $G_{x_0}$ (and hence in $A$). In step 3, we give an existence of a dense set $W$ (of $A$) whose  elements have $k$-th roots in $A$ as required.

\medskip
\noindent{\bf Step 1:}
Since  $S(L(\mathbb R))$ is dense in $L(\mathbb R)$ by Proposition \ref{B}, and $A\cap L(\mathbb R)$ is open in $L(\mathbb R)$,
$S(A\cap L(\mathbb R))$ is dense in $A\cap L(\mathbb R)$.
Since $G(\mathbb R)^*\subset A\subset G(\mathbb R)$, we have
$L(\mathbb R)^*\subset A\cap L(\mathbb R)\subset L(\mathbb R)$.
Note that
if $(k,\circ(A/G(\mathbb R)^*))=1$,
then $(k,\circ((A\cap L(\mathbb R))/L(\mathbb R)^*))=1$.
Also $P_k(G(\mathbb R)^*)$ is dense in $G(\mathbb R)^*$, 
which implies that $P_k(L(\mathbb R)^*)$ is dense in $L(\mathbb R)^*$.
Now by Lemma \ref{L1} applied to the reductive group $L(\mathbb R)^*$, we have
$P_k:S(L(\mathbb R)^*)\to S(L(\mathbb R)^*)$ is surjective.
So $P_k:S(A\cap L(\mathbb R))\to S(A\cap L(\mathbb R))$ is surjective by \cite[Theorem 5.5]{Ch1}. Hence $P_k(A\cap L(\mathbb R))$ is dense in $A\cap L(\mathbb R)$.

\medskip
\noindent{\bf Step 2:}
For a given $x_0\in S(A\cap L(\mathbb R))$,  by Lemma \ref{surjectivity of abelian subgroup}, 
there exists a closed abelian subgroup $B_{x_0}$ of $A\cap L(\mathbb R)$ containing $x_0$ and $P_k(B_{x_0})=B_{x_0}$. 
Also, $B_{x_0}$ has finitely many connected components in real topology. 
Now consider the subgroup $G_{x_0}=B_{x_0}\ltimes R_u(G)(\mathbb R)$ of $A$.
Note that the unipotent radical $N=R_u(G)(\mathbb R)$ is a simply connected nilpotent Lie group. 
Let $$N=N_0\supset N_1\supset\cdots\supset N_r=\{e\}$$
be the series as in Theorem \ref{DaM}. Let $V_j:=N_j/N_{j+1}$ for $j=0,1,\ldots,r-1.$ As $N$ is simply connected, $V_j$'s are all finite dimensional real vector spaces, and $B_{x_0}$-action on $V_j$'s are irreducible.

Now
for $j=0,1,\ldots,r-1$,
let $F_j$ be the set of elements of $B_{x_0}$ which act trivially on $V_j$ under conjugation action.
Note that $F_j$ is a closed subgroup of $B_{x_0}$ for all $j.$
We shall consider the following two cases separately.

i) ${\rm dim}(F_j)<{\rm dim}(B_{x_0})$ for all $j$,

ii) for some $j$, ${\rm dim}(F_j)={\rm dim}(B_{x_0}).$

Suppose $(i)$ holds. Then $U':=B_{x_0}-\cup_jF_j$ is a dense open set in $B_{x_0}.$ Therefore
all elements in $U'$ act non trivially on all $V_j$'s, and hence by Theorem \ref{DaM}, $gN\subset P_k(G_{x_0})$
for all $g\in U'.$ If we take $W_{x_0}=U'\times N$, then $W_{x_0}$ is a dense open subset of $G_{x_0}$ such that $W_{x_0}\subset P_k(G_{x_0}).$

Now suppose $(ii)$ holds. Let $\mathcal I\subset\{0,1,\ldots,r-1\}$ be the set of indices such that  ${\rm dim}(F_j)={\rm dim}(B_{x_0})$ for all $j\in \mathcal I$. Then $\cup_{i\notin\mathcal{I}}F_i$ is a
proper closed analytic subset of $B_{x_0}$ of smaller dimension. Note that $B_{x_0}\setminus \cup_{i\notin\mathcal{I}}F_i$ is a dense open set in $B_{x_0}$.\\

 {\underline {Claim:}} the coset $xN\subset P_k(G_{x_0})$ for all $x\in B_{x_0}\setminus \cup_{i\notin\mathcal{I}}F_i$.\\
 
If ${\rm dim}(F_j)={\rm dim}(B_{x_0})$,
then $F_j$ is the union of some connected components of $B_{x_0}$ (containing the identity component of $B_{x_0}$). 
Fix $j\in\mathcal {I}$.
We will show that for any $x\in F_j\setminus \cup_{i\notin\mathcal{I}}F_i$ the coset $xN\subset  P_k(G_{x_0})$. 

It might happen that $x$ belongs to $F_s$ (for $s\in\mathcal I$) other than $F_j$. Let $\mathcal{J}=\{m\in \mathcal{I}|x\in F_m\}$. Clearly, $j\in\mathcal J$ and $x\in \cap_{i\in \mathcal {J}}F_i\subset F_j$. To prove the coset $xN$ has a $k$-th root, we will apply Theorem \ref{DaM}, and for that we will show the existance of a $k$-th root of $x$ in $\cap_{i\in \mathcal {J}}F_i$.

Since $P_k(B_{x_0})=B_{x_0}$, we get $P_k:B_{x_0}/B_{x_0}^*\to B_{x_0}/B_{x_0}^*$ is surjective, and hence 
$k$ is co-prime to the order of the component group $\Gamma=B_{x_0}/B_{x_0}^*$. Let the cardinality
of $\Gamma$ be $n$. Then there exist integers $a$ and $b$ such that $ak+bn=1$. Without loss of generality, we can choose $a$ to be a positive integer.
Note that $x=x^{ak+bn}=(x^a)^k(x^n)^b$ where $x^n\in B_{x_0}^*$ and so does $(x^n)^b$. As $B_{x_0}^*$ is a connected abelian group,
it is divisible. So there exists $x'\in B_{x_0}^*$ such that $x'^k=(x^n)^b$. This gives $x=(x^a)^kx'^k=(x^ax')^k$
as both $x^a$ and $x'$ commute with each other.
Clearly $x^ax'\in\cap_{i\in \mathcal J }F_i\subset F_j$. 
This implies that for $l\in\{0,1,\ldots,r-1\}$, whenever the action of $x$ on $V_l$ is trivial there exists a $k$-th root of $x$ (namely $x^ax'$) which acts trivially on $V_l$. Therefore by Theorem \ref{DaM} we conclude that $xn\in  P_k(G_{x_0})$ for all $n\in N$.

Finally, since $j\in\mathcal {I}$ is arbitrary, this proves the claim.
Hence, there exists a dense open
set $W'$ in $B_{x_0}$ such that $W'_{x_0}\subset P_k(G_{x_0})$ where $W'_{x_0}=W'\times N.$ 
Thus for both cases we have a dense open set $W_{x_0}$ in $G_{x_0}$ such that each element of $W_{x_0}$ has a $k$-th root in $G_{x_0}$, and hence in $A$.

\medskip
\noindent{\bf Step 3:}
We note that $S(A\cap L(\mathbb R))N=\cup_{x_0\in S(A\cap L(\mathbb R))}G_{x_0}$. Now set $$W=\cup_{x_0\in S(A\cap L(\mathbb R))} W_{x_0}.$$ Then $W$ is a dense set in $A$ such that each element of $W$ has a $k$-th root in $A$, which proves $(2)$ implies $(1)$.

$(2)\Leftrightarrow (3)$ Follows from Lemma \ref{L1}.

$(3)\Leftrightarrow (4)$ Follows from \cite[Theorem 5.5]{Ch1}.

A. Borel and Tits showed that if G is a Zariski connected algebraic group defined over $\mathbb R$, then either
$G(\mathbb R)=G(\mathbb R)^*$ or $G(\mathbb R)/G(\mathbb R)^*$ is a direct product of cyclic groups of order two
(see \cite[Theorem 14.4]{BT}).
Hence, the result follows immediately from the above.
\qed

\begin{remark}
Let $G$ be an algebraic group over $\mathbb{C}$, which is not necessarily Zariski-connected. 
	Let $G(\mathbb C)$ denote the complex point of $G$.
	Let $k\in\mathbb N$. Then it is immediate that
	$P_k(G(\mathbb C))$ is dense in $G(\mathbb C)$ if and only if $k$ is co-prime to the order of $G/G^0$.
\end{remark}

Theorem \ref{T1} asserts the following corollary.

\begin{corollary}\label{C}
	Let $G$ be as in Theorem \ref{T1}. Let $G(\mathbb R)=L(\mathbb{R})\ltimes R_u(G)(\mathbb R)$ and $k\in\mathbb N$.
	Then $P_k(G(\mathbb R))$ is dense if and only if $P_k(L(\mathbb R))$ is dense.
\end{corollary}
\begin{proof}
	The proof follows using the same procedure as in the proof of Theorem \ref{T1} applied to the group $A=G(\mathbb R)$.
\end{proof}

\section{Application to weak exponentiality}

For a Lie group $G$ with Lie algebra ${\rm Lie}(G)$, let $\exp:{\rm Lie}(G)\to G$ be the exponential map of $G$.
We recall that $G$ is {\it weakly exponential} if $\exp({\rm Lie}(G))$ is dense in $G$. The group $G$ is said to 
{\it exponential}
if $\exp({\rm Lie}(G))=G$.

\begin{remark}\label{C1}
	 In view of $(1)$ $\Rightarrow$ $(2)$
	in Theorem \ref{T1}, we observe that
	$P_k(A)$ is dense implies $k$ is co-prime to the order of $A/G(\mathbb R)^*$ and  $P_k(G(\mathbb R)^*)$ is
	dense. By \cite[Corollary 1.3]{BM}, it follows that $P_k(A)$ is dense for all $k$ if and only if $A$ is weakly 
	exponential.
\end{remark}

Remark \ref{C1} can be thought of as the analogous result of Hofmann and Lawson 
\cite[Proposition 1]{HL}, which states the following:
For a closed subgroup $H$ (possibly disconnected) of a connected Lie group, 
$H$ is exponential if and only if $H$ is divisible, i.e., $P_k(H)=H$ for all $k\in\mathbb N$.

\begin{remark}\label{C2}
	Let $G$ be a Zariski-connected complex algebraic group defined over $\mathbb R$. Then the following are equivalent:
	(i) the map $P_2:S(G(\mathbb R))\to S(G(\mathbb R))$ is surjective,
	(ii) the image of the map $P_2:G(\mathbb R)\to G(\mathbb R)$ is dense,
	(iii) $G(\mathbb R)$ is weakly exponential.
	This can be seen from \cite[Theorem 1.6]{Ch1} and the fact that $G(\mathbb R)/G(\mathbb R)^*$ is a
	group of order $2^m$ for some $m$. It can also be deduced from Theorem \ref{T1}. Indeed, by  \cite[Theorem 5.5]{Ch1}
	and Lemma \ref{L1}, 
	statement (i) implies $2$ is co-prime to the order of the group $G(\mathbb R)/G(\mathbb R)^*$ and $P_2(G(\mathbb R)^*)$ is dense.
	Thus by Theorem \ref{T1}, $P_2(G(\mathbb R))$ is dense. Now (ii) $\Leftrightarrow$ (iii) follows from
	Remark \ref{C1}.
\end{remark}
Note that if we replace $G(\mathbb R)$ by $G(\mathbb R)^*$ in statements (i)-(iii) of Remark \ref{C2}, then the equivalence follows from \cite[Theorem 1.6]{Ch1}.
Also, Remark \ref{C2} generalizes the result for a connected linear group (see \cite[Corollary 1.5]{BM}). Note that
\cite[Corollary 1.5]{BM} (see Proposition \ref{P} below) can be proved using results different from theorems in \cite{BM}.
For example, we now provide a proof which is due to P. Chatterjee.

\begin{proposition}\label{P}
	Let $G$ be a connected linear Lie group. Then $G$ is weakly exponential if and only if $P_2(G)$ is dense.
\end{proposition}
\begin{proof}
	We assume that $P_2 (G)$ is dense in $G$.
	Since $G$ is linear so is $G/{\rm Rad}(G)$ (see p. 26.
	\cite[Proposition 5.2]{OV}). Now as
	$G/{\rm Rad}(G)$ is a connected linear semisimple Lie group, it is isomorphic to
	$A(\mathbb R)^*$ for some Zarsiki connected (semisimple)
	algebraic group $A$ defined over $\mathbb R$.
	Since $P_2(G)$ is dense in $G$, it follows that $P_2 (A (\mathbb R)^*)$ is dense in $A (\mathbb R)^*$. Using 
	Theorem \ref{T1}, we see that $P_2:S(A(\mathbb R)^*) \to S(A(\mathbb R)^*)$ is surjective.
	Now apply \cite[Theorem1.6]{Ch1} to
	see that $A(\mathbb R)^*$ is weakly exponential. Thus $G/{\rm Rad}(G)$ is weakly exponential. As ${\rm Rad}(G)$ is
	connected solvable, 
	it is weakly exponential and hence by \cite[Lemma 3.5]{HM} $G$ is weakly exponential. The other part is obvious.
\end{proof}

The following corollary establishes a special behaviour for a dense image of the power map on $G(\mathbb R)$,
analogous to a result of \cite[Corollary 2.1A]{HM} for the exponential map. Moreover, 
it generalizes the result \cite[Corollary 2.1A]{HM} and \cite[Proposition 3.3]{BM}
restricted to the group $G(\mathbb R)$. 

\begin{corollary}\label{T3}
	Let $G$ be a Zariski-connected complex algebraic group defined over $\mathbb R$. Let $N$ be a Zariski-connected, Zariski-closed
	algebraic normal subgroup of $G$ defined over $\mathbb R$. If both $P_k(G(\mathbb R)/N(\mathbb R))$ and $P_k(N(\mathbb R))$ are dense,
	then $P_k(G(\mathbb R))$ is dense.
\end{corollary}

\begin{proof}
As $G$ is Zariski connected for any odd $k\in\mathbb N$, 
$P_k(G(\mathbb R)/N(\mathbb R))$, $P_k(N(\mathbb R))$ and 
$P_k(G(\mathbb R))$ are dense by Theorem \ref{T1}.

Suppose both the images of $P_2:G(\mathbb R)/N(\mathbb R)\to G(\mathbb R)/N(\mathbb R)$ 
and $P_2:N(\mathbb R)\to N(\mathbb R)$ are dense. Then by Remark \ref{C2}, $G(\mathbb R)/N(\mathbb R)$ and
$N(\mathbb R)$ are weakly exponential and hence connected.
Recall that for any closed subgroup $L$ of a topological group $H$, $H$ is connected if $H/L$ and $L$ are connected.
This implies $G(\mathbb R)=G(\mathbb R)^*$.
Since both the groups $G(\mathbb R)^*/N(\mathbb R)^*$
and $N(\mathbb R)^*$ are weakly exponential, by \cite[Corollary 2.1A]{HM}, we have $G(\mathbb R)^*$ is weakly exponential.
Therefore $P_2(G(\mathbb R))$ is dense
by Remark \ref{C2}.
\end{proof}
\begin{remark}\label{Regular dense}
	Suppose $G$ is a Zariski connected complex algebraic group defined over $\mathbb R$ and $k\in\mathbb N$. Then 
	$P_k$ is surjective on $G(\mathbb R)/G(\mathbb R)^*$ and
	$P_k({\rm Reg}(G(\mathbb R)^*))\supset{\rm Reg}(G(\mathbb R)^*)$ together imply
	$P_k({\rm Reg}(G(\mathbb R)))\supset {\rm Reg}(G(\mathbb R))$.
\end{remark}

From Remark \ref{Regular dense} we have the following.

\begin{remark}\label{Dani}
	For a Zariski connected complex algebraic group $G$ defined over $\mathbb R$, the density of image $P_k(G(\mathbb R))$ is equivalent
	to the fact that the complement of the image is of measure zero. This follows since the density of image $P_k(G(\mathbb R))$ implies $\rm{Reg} (G)\cap G(\mathbb R)\subset P_k(G(\mathbb R))$ ($\rm{Reg}(G)$ is non-empty Zariski-open in $G$).
\end{remark}

\section{Surjectivity of the power maps}

\noindent{\it Proof of Corollary~\ref{C3}}:
By \cite[Lemma 5.6]{Ch1}, $P_k:G(\mathbb R)\to\ G(\mathbb R)$ is surjective if and only if for every unipotent element
$u\in G(\mathbb R)^*$, the map
$P_k:S(Z_{G(\mathbb R)}(u))\to\ S(Z_{G(\mathbb R)}(u))$ is surjective.
Note that $Z_{G(\mathbb R)}(u)=Z_G(u)(\mathbb R)$. For each unipotent element $u\in G(\mathbb R)^*$, let
$H_u=Z_G(u)$. Since $P_k$ is surjective on $S(H_u(\mathbb R))$,
by \cite[Theorem 5.5]{Ch1}, $P_k:H_u(\mathbb R)/H_u(\mathbb R)^*\to H_u(\mathbb R)/H_u(\mathbb R)^*$ 
and $P_k:S(H_u(\mathbb R)^*)\to S(H_u(\mathbb R)^*)$ are surjective. By Lemma \ref {L1},  
$P_k(H_u(\mathbb R)^*)$ is dense in
$H_u(\mathbb R)^*$, which implies $P_k(H_u(\mathbb R))$ is dense in
$H_u(\mathbb R)$ by Theorem \ref{T1}. This completes the proof.
\qed\\

The following result is well known (see \cite[Theorem 2.2]{DT}) and was later deduced by P. Chatterjee (see \cite[Corollary 5.7]{Ch1}). Here we prove it as an application of Corollary \ref{C3}.

\begin{corollary}\label{C14}
	Let $G$ be a connected real algebraic group. Then $G(\mathbb R)^*$ is exponential if and only if 
	$Z_{G(\mathbb R)^*}(u)$ is weakly exponential for
	all unipotent elements $u\in G(\mathbb R)^*$.
\end{corollary}
\begin{proof}
	By McCrudden's criterion, $G(\mathbb R)^*$ is exponential if and only if $P_k:G(\mathbb R)^*\to G(\mathbb R)^*$ is surjective
	for all $k\in\mathbb N$. Further, by Corollary \ref{C3}, 
	$P_k:G(\mathbb R)^*\to\ G(\mathbb R)^*$ is surjective if and only if
	$P_k(Z_{G(\mathbb R)^*}(u))$ is dense in $Z_{G(\mathbb R)^*}(u)$ for all unipotent elements $u\in G(\mathbb R)^*$.
	Now by Remark \ref{C1}, we conclude that $Z_{G(\mathbb R)^*}(u)$ is weakly exponential.
\end{proof}

{\bf Acknowledgement.} 
I would like to thank P. Chatterjee for raising the question about dense image of the power map 
for a disconnected group, especially the statement of Theorem \ref{T1}. 
He has also provided various comments and 
suggestions for other parts of the paper, and communicated the proof of Proposition \ref{B}.
I thank U. Hartl for pointing out the reference \cite{HP}.
 I would also like to thank S. G. Dani and Riddhi Shah for many comments on the manuscript. I sincerely appreciate the anonymous referee for useful comments and suggestions.
I would like to thank Indian Statistical Institute Delhi, India for a post doctoral fellowship while most of this work was done.

\vskip2mm

\begin{flushleft}
	\begin{flushleft}
		Arunava Mandal \\
		Theoretical Statistics and Mathematics Unit,\\
		Indian Statistical Institute, Bangalore Centre,\\
		Bengaluru 560059, India.

		E-mail: {\tt a.arunavamandal@gmail.com} 
	\end{flushleft}
\end{flushleft}

\end{document}